\documentclass[12pt,twoside]{article}
\usepackage{amssymb}
\usepackage[mathscr]{eucal}
\usepackage{eufrak}
\usepackage{amsmath,amsthm}
\usepackage{mathrsfs}
\usepackage[colorlinks=true,
   urlcolor=blue,           filecolor=green,      
   citecolor=green,      
   linkcolor=red,           bookmarks=true,
  unicode,
   plainpages=false,   ]{hyperref}
\usepackage{color}
\usepackage[latin1]{inputenc}
\DeclareMathOperator*{\esssup}{ess\,sup}
\def\bp{\begin{proof}}
\def\ep{\end{proof}}

\def\n{\nabla}
\def\intl#1{\int\limits_{#1}}
\def\intll#1#2{\int\limits_{#1}^{#2}}
\def\dm{|\hskip-0.05cm|}
\def\OO{\Omega}
\def\displ{\displaystyle}
\def\VSE{\vspace{6pt}\\&\displ }
\def\VS{\vspace{6pt}\\\displ }
\def\rf#1{{\rm(\ref{#1})}}
\def\chiu{\hfill$\displaystyle\vspace{4pt}
\underset\Box\null$\par}
\def\R{\Bbb R}
\def\N{\Bbb N}

\def\à{\`{a}}

\def\dy{\displaystyle}
\def\vep{\varepsilon}

\def\be{\begin{equation}}
\def\ba{\begin{array}}
\def\ea{\end{array}}
\def\ee{\end{equation}}
\def\vs1{\vspace{1ex}}

\def\ov{\overline}
\def\po{\partial\Omega}

\def\é{\'{e}}
\font\sc=cmcsc10
\title{\large Some new   properties of a suitable weak solution\\ to the Navier-Stokes equations}
\author{\sc  Francesca Crispo\thanks{Dipartimento di Matematica e Fisica,  
Universit\`{a} degli Studi della Campania
``L. Vanvitelli'', via Vivaldi 43, 81100 Caserta,
 Italy.
 francesca.crispo@unicampania.it}
\and \sc
Carlo Romano Grisanti \thanks{Dipartimento di Matematica, Universit\`a di Pisa, via Buonarroti 1/c, 56127 Pisa, Italy. carlo.romano.grisanti@unipi.it}\and \sc  Paolo Maremonti
\thanks{
Dipartimento di Matematica e Fisica,  
Universit\`{a} degli 
Studi della Campania
``L. Vanvitelli'', via Vivaldi 43, 81100 Caserta,
 Italy.
paolo.maremonti@unicampania.it}\thanks{The research activity of C.R. Grisanti is performed under the auspices
of  GNAMPA-INdAM and partially supported by the Research
Project of the University of Pisa "Energia e regolarità: tecniche
innovative per problemi classici di equazioni alle derivate parziali". The research activity of F. Crispo and P. Maremonti is performed under the
auspices of   GNFM-INdAM and is partially supported by MIUR via the PRIN 2017 ``Hyperbolic Systems of Conservation Laws and Fluid Dynamics: Analysis and Applications''.}}
\date{\small \it In Memoria di Christian Simader}
\begin{document}
\markboth{\footnotesize\rm Francesca Crispo, Carlo Romano Grisanti and   P.
Maremonti} {\footnotesize\rm
Some new properties of a suitable  weak solution to the Navier-Stokes equations}
\maketitle 
{\bf Abstract} - {\small The paper is concerned with 
the IBVP of  the Navier-Stokes equations. The goal is the construction of a weak solution enjoying some new properties. Of course, we look for   properties which are global in time.  The results hold assuming an initial data      $v_0\in J^2(\OO)$.} 
\vskip 0.2cm
 \par\noindent{\small Keywords: Navier-Stokes equations,   weak solutions, regularity and partial regularity. }
  \par\noindent{\small  
  AMS Subject Classifications: 35Q30, 35B65, 76D03.}  
 \par\noindent
 \vskip -0.7true cm\noindent
\newcommand{\red}{\protect\bf}
\renewcommand\refname{\centerline
{\red {\normalsize \bf References}}}
\newtheorem{ass}
{\bf Assumption} 
\newtheorem{defi}
{\bf Definition} 
\newtheorem{tho}
{\bf Theorem} 
\newtheorem{rem}
{\sc Remark} 
\newtheorem{lemma}
{\bf Lemma} 
\newtheorem{coro}
{\bf Corollary} 
\newtheorem{prop}
{\bf Proposition} 
\renewcommand{\theequation}{\arabic{equation}}
\setcounter{section}{0}
\section{Introduction}\label{intro} This note concerns the 3D-Navier-Stokes initial boundary 
value problem:
\be\label{NS}\ba{l}v_t+v\cdot
\nabla v+\nabla\pi_v=\Delta
v,\;\nabla\cdot
v=0,\mbox{ in }(0,T)\times\OO,\VS v=0\mbox{ on }(0,T)\times\po,\hskip0.12cm
v(0,x)=v_0(x)\mbox{ on
}\{0\}\times\OO.\ea\ee  In system \rf{NS} $\OO\subseteq\R^3$ is assumed bounded or exterior, and its boundary is  smooth. The symbol   $v$ denotes the kinetic field, $\pi_v$ is the pressure field,
 $v_t:=
\frac\partial{\partial t}v$  and
 $v\cdot\nabla v:=
v_k\frac\partial{\partial x_k}v$. 
In several papers, related to the Navier-Stokes initial boundary value problem,  the authors  give results concerning the partial regularity of  a suitable weak solution (see Definition\,\ref{SWS} below).  This  is made in order    to highlight the  properties of a weak solution, corresponding to a data $v_0\in L^2(\OO)$, divergence free, that can be suitable  to state the well posedness of the equations, see e.g. \cite{L, Ld,  CKN, Du, Vs, Ch-Vs,CM-I,CM-II,Mi,CM-III}\footnote{\,There is a wide literature concerning extension to the IBVP of results proved for the Cauchy problem. One of the most interesting of these kinds of extensions is sure the energy inequality in strong form, see e.g. \cite{GM,MS}}. We believe that, in connection with the non-well posedeness of the Navier-Stokes Cauchy or IBVP problem, this kind of investigation achieves   a further interest. Actually, in the  recent paper \cite{BV}, it is considered the possibility of non uniqueness of a weak solution to the Navier-Stokes equations. This is proved for {\it very weak solutions}, that is   solutions satisfying a variational formulation of the Navier-Stokes equations and simply belonging to $C([0,T);L^2(\OO))$. As a consequence of the weakness of the solutions, the result of non uniqueness   fails to hold for regular solutions, but {\it a priori} it also does not work  for a suitable weak solution, that is a solution  verifying an energy inequality. So that, in order to better delimit  the validity of a possible  counterexample to the uniqueness in the set of weak solutions corresponding to an initial data in $L^2(\OO)$, it seems of a certain interest to support  the energy inequality, or its variants,  by means of a wide set of global properties of the weak solutions not necessarily only consequences of the energy inequality, but of the coupling of other {\it a priori} estimates.  \par 
The aim of this note is to prove some new properties of a weak solution. We investigate two questions. One is related to a \underline {sort of energy  equality} for a suitable weak solution. It is easy to understand that the possible validity of the energy equality achieves a {\it mechanical} interest that goes beyond the  above question concerning the well posedeness. Actually,  we construct a weak solution $(v,\pi_v)$ to the Navier-Stokes initial boundary value problem such that the ``energy equalities'' of the kind\be\label{EE-I} \dm v(t)\dm_2^2+2\intll st\dm \n v(\tau)\dm_2^2d\tau-\dm v(s)\dm_2^2=-H(t,s), \mbox{a.e. in }t\geq s>0\mbox{ and for }s=0\, \ee and  
\be\label{EE-II}  {\displ2\intll st\dm \n v(\tau)\dm_2^2d\tau} = F(t,s)\big({\dm v(s)\dm_2^2-\dm v(t)\dm_2^2}\big), \mbox{a.e. in }t\geq s>0\mbox{ and for }s=0\, \ee are fulfilled.
 The functions  $H(t,s)$ and $F(t,s)$ have  suitable expressions, see formula \rf{SEE}  and formula \rf{SEE-I}. If $H(t,s)\leq0$, then the energy equality holds (that is  {\it a fortiori} $H(t,s)=0$). If $F(t,s)\geq1$, then the energy equality holds (that is {\it a fortiori} $F(t,s)=1$).   These results are a consequence of the fact that we are able to prove that an approximating  sequence $\{(v^m,\pi_{v^m})\}$ is strongly converging in $L^r(0,T;W^{1,2}(\OO))$ for all $r\in[1,2)$. The strong convergence, in turn, is a consequence of the property: $P\Delta v^m\in L^{\frac23}(0,T;L^2(\OO))$ for all $m\in\N$ and  $T>0$.  Unfortunately we are not able to put $r=2$, that should give the energy equality. For 2D-Navier-Stokes equations one proves that  $H(t,s)=0$. It is important to stress that the term $H(t,s)$ is equal to zero in 2D-case  thanks to our approximating approach, without appealing to the regularity of the limit. Another result proves that $v\in L^{\mu(p)}(0,T;L^p(\OO))$, with $\mu(p):=\frac p{p-2}$ and $p\in(6,\infty]$. This result is not new in literature. A first contribution in this sense is proved in \cite{T} for a particular geometry and it is reconsidered in \cite{Du}. The proof given in \cite{Du} for exterior domains is not completely clear to the present authors. However our proof is alternative with respect to the ones of the  quoted papers. \par   \par
 In order to better state our result we recall the following definitions. 
We denote by $J^2(\OO)$ and $J^{1,2}(\OO)$ the completion of $\mathscr C_0(\OO)$ in $L^2(\Omega)$ and in $W^{1,2}(\OO)$ respectively, where $\mathscr C_0(\OO)$ is the set of smooth divergence free functions. Moreover $(\cdot,\cdot)$ represents the scalar product in $L^2(\Omega)$. 
 \begin{defi}\label{WS}{\sl  Let $v_0\in J^2(\OO)$. A pair $(v,\pi_v)$, such that $v:(0,\infty)\times\OO \to\R^3$ and $\pi_v:(0,\infty)\times\OO \to \R$, is said a weak solution to problem {\rm\rf{NS}} if
\begin{itemize}\item [\label{IO}i)] for all $T>0$,
 $v\in  L^2(0,T; J^{1,2}
(\OO ))$ and, for some $q,r$, $\pi_v\in L^r_{\ell oc}([0,T);L^q_{\ell oc}(\ov\OO))$,
\item
[\label{II}ii)] $\displ\lim_{t\to0}\dm
v(t)-v_0\dm_2=0$,\item[iii)]
for all $t,s\in(0,T)$,  the pair $(v,\pi_v)$
satisfies the  equation:
$$\displ\intll
st\Big[(v,\varphi_\tau)-(\nabla
v,\nabla
\varphi)+(v\cdot\nabla\varphi,v)+(\pi_v,\nabla
\cdot\varphi)\Big]d\tau+(v(s),\varphi
(s))=(v(t),\varphi(t)),$$
$$\mbox{ for all } \varphi\in C^1_0([0,T)\times\OO ).$$
\end{itemize}}\end{defi} 
In \cite{CKN} and in \cite{scheffer}, in order to investigate  the regularity of a weak solution, it is introduced an     energy inequality having a local character:
\begin{defi}\label{SWS}{\sl A pair $(v,\pi_v)$ is said a suitable weak solution if it is a weak solution in the sense of the Definition\,\ref{WS} and, moreover,
\be\label{SEI}\ba{l}\displ\intl{\OO}|v(t)|^2\phi(t)dx+2\intll s t\intl
{\OO}|\nabla v|^2\phi\, dxd\tau\leq \intl{\OO}|v(s)|^2\phi(s)dx\VS\hskip 2,5cm+\intll s t\intl{\OO}|v|^2(\phi_\tau+\Delta\phi)dxd\tau+\intll s t\intl{\OO}(|v|^2+2\pi_v)
v\cdot\nabla\phi dxd\tau,\ea\ee for all $t> s$, for $s=0$ and a.e. in $s\geq 0$, and for all 
nonnegative $\phi\in C_0^\infty(\R\times\ov\OO)$. We denote by $\varSigma\subseteq[0,\infty)$ the set of the instants $s$ for which inequality  \rf{SEI} holds.}\end{defi} 
Thanks to the properties of the pressure field furnished by the existence theorem, from inequality \rf{SEI} one deduces the classical one:
\be\label{EI}\dm v(t)\dm_2^2+2\intll st\dm \n v(\tau)\dm_2^2d\tau\leq \dm v(s)\dm_2^2,\mbox{ for all }t>s\mbox{ and  }s\in\varSigma\,.\ee
 We are going to prove the following result.
  \begin{tho}\label{CT}{\sl For all $v_0\in J^2(\OO)$ there exists a suitable weak solution $(v,\pi_v)$ to problem \rf{NS} that is the weak limit in $L^2(0,T; J^{1,2}(\Omega))\times L^q_{\ell oc}([0,T);L^2_{\ell oc}(\ov\OO))$, $q\in(1,\frac{12}{11})$, of a sequence $\{(v^m,\pi_{v^m})\}$ of solutions to \eqref{MNS}.
The sequence $\{v^m\}$ converges strongly to $v$ in $L^p(0,T;W^{1,2}(\OO))$ for all $p\in [1,2)$. Further, for any $q\in(6,\infty]$,  $v\in L^{\mu(q)}(0,T;L^q(\OO))$ with $\mu(q):=\frac q{q-2}$, and
 $v$ satisfies relation  \rf{EE-I} with 
\be\label{SEE}H(t,s):=\left\{\hskip-0.2cm\ba{l}\displ\lim_{\alpha\to0} \lim_{m\to\infty}\alpha\!\intll st\! \frac{\|v^m(\tau)\dm_2^2}{\left(K+\|\nabla v^m(\tau)\|_2^2\right)^{\alpha+1}}\frac d{d\tau}\|\nabla v^m(\tau)\|_2^2\,d\tau,\mbox{ for  }s>0,\VS\lim_{s\to0}\lim_{\alpha\to0} \lim_{m\to\infty}\alpha\!\intll st \!\frac{\|v^m(\tau)\dm_2^2}{\left(K+\|\nabla v^m(\tau)\|_2^2\right)^{\alpha+1}}\frac d{d\tau}\|\nabla v^m(\tau)\|_2^2\,d\tau,\mbox{ for  }s=0\,,\ea\right.\ee 
for any arbitrary constant $K>0$, as well as 
 $v$ satisfies relation  \rf{EE-II} with 
\be\label{SEE-I}F(t,s):=\lim_{\alpha\to0}\lim_{m\to\infty}\frac 1{\big(K_1+\dm \n v^m(t_{\alpha,m})\dm_2\big)^\alpha}\ee 
for any arbitrary constant $K_1\geq0$. Finally,  the following inclusion holds:  $\mathcal G_1:=\{t,s\mbox{ such that }\rf{EE-I}\mbox{ is true}\}\subseteq\mathcal G_2:=\{t,s\mbox{ such that }\rf{EE-II}\mbox{ is true}\}$.}
\end{tho}
\begin{rem}{\rm Since the proprieties of the pressure field $\pi_v$ are not our main interest in this paper, we limit ourselves to point out the one that allows us to state that $(v,\pi_v)$ is a suitable weak solution. Actually, in our construction  the pressure $\pi_v$  enjoys the properties that one can deduce by means of Lemma\,\ref{APWS}.  For more exhaustive properties relative  to the pressure field of a suitable weak solution (that is with an initial data only in $J^2(\OO)$) a possible reference is \cite{JM}.\par 
We note that the quantity $H(t,s)$ is independent of the constant $K$. This fact is intriguing and somehow leads to conjecture that $H(t,s)=0$.
\par If $v_0\in J^2(\OO)\setminus J^{1,2}(\OO)$, almost everywhere in $t>0$, following the proof idea we also get $$\dm v(t)\dm_2^2+2\intll0t\dm \n v(\tau)\dm_2^2d\tau=-\lim_{\alpha\to0}\lim_{m\to\infty}\intll0t\frac{\|v^m(\tau)\dm_2^2}{\left(K+\|\nabla v^m(\tau)\|_2^2\right)^{\alpha+1}}\frac d{d\tau}\|\nabla v^m(\tau)\|_2^2\,d\tau\,. $$  \par One proves that there exists an instant $\theta>0$ such that $ v^m(t,x)\in C([\theta,\infty);J^{1,2})$, in particular there exists a $M$ such that $\dm \n v^m(t)\dm_2\leq M$ for all $t\geq \theta$ and $m\in\N$ (one proves this result by repeating the arguments employed for the structure theorem by Leray). Hence, via estimate \rf{ddtnablavm-I} and taking into account the energy inequality, we can deduce $$\mbox{for all }t>\theta, \; \displ\lim_{\alpha\to0} \lim_{m\to\infty}\alpha\!\intll \theta t\! \frac{\|v^m(\tau)\dm_2^2}{\left(K+\|\nabla v^m(\tau)\|_2^2\right)^{\alpha+1}}\frac d{d\tau}\|\nabla v^m(\tau)\|_2^2\,d\tau=0\,.$$  Hence we get that function  $H(t,s)=H(\theta,s)$ for all $t>\theta$, and $H$ becomes a constant function for $t>\theta$.\par Concerning the function $F$ we remark that its values are independent of $K_1\geq0$. The fact that $K_1$ can be chosen equal to zero makes a difference with $K$ in function $H$, as well as  $\mathcal G_1\subseteq \mathcal G_2$ is another difference.\par   In the introduction we remarked that if  $F(t,s)\geq1$, then the energy equality holds, that is $F(t,s)=1$. Since we are not in a position  to prove $F(t,s)\geq1$, {\it a priori} we have to consider that $F(t,s)\leq1$. However, we can claim that almost everywhere in $t>0$,   $F(t,0)>0$ holds. Actually, more in general, assume that $s\in \Sigma$ and $\dm v(s)\dm_2\ne0$ and that exists  a sequence $\{t_p\}$ converging to $s$ such that $F(t_p,s)=0$. Then, from formula \rf{EE-II}, we deduce that    $\dm v(\tau)\dm_2=0$ a.e. in $\tau\in(s,t_p)$ holds for all $p\in\N$. Hence  we can select a new sequence $\{t'_p\}\subset(s,t_p)$ such that $\dm v(t'_p)\dm_2=0$. By virtue of the right-$L^2$-continuity in $s$, we get $\displ\lim_{t'_p\to s}\dm v(t'_p)-v(s)\dm_2=0$, which is a contradiction with $\dm v(s)\dm_2\ne0$.\par  From formulas \rf{EE-I}-\rf{EE-II}, a.e. in $t\in \mathcal G_1$ and for $s=0$, we easily deduce that $$H(t,0)=(1-F(t,0))\big(\dm v(0)\dm_2^2-\dm v(t)\dm_2^2\big)\,.$$ Therefore, via \rf{EE-I} we deduce \be\label{CF}H(t,0)=\big(\mbox{$\frac1{F(t,0)}-1$}\big)\intll 0t\dm\n v(\tau)\dm_2^2d\tau\,.\ee Recalling that,  for $t\geq\theta$, $H(t,0)=H(\theta,0)$, then from \rf{CF} we deduce that $F(t,0)$ is a continuous function for $t\geq\theta$. }
\end{rem} 
\begin{rem}{\rm In paper \cite{Ng}  a new energy inequality  is proposed:
$$\dm v(t)\dm_2^2+N(t)+2\intll0t\dm \n v(\tau)\dm_2^2d\tau\leq \dm v_0\dm_2^2,\mbox{ for all }t>0\,.$$  Function $\displ N(t):=\limsup_{\delta\to0}\intll\delta t \left\| \frac{u(\tau)-u(\tau-\delta) }{\delta^\frac12}\right\|_2^2d\tau\geq0$ can be intepretred as $\frac12$-time derivative. It is not known if $N(t)>0$ holds. In the two dimensional case one proves that $N(t)=0$. Of course,   we are not able to compare the solution furnished in \cite{Ng} and the one of Theorem\,\ref{CT}. \par In paper \cite{Mi}  the compatibility between an energy equality and an initial data $v_0\in J^2(\OO)$  is proved. This supports the idea that $H(t,s)$ can be equal to zero.}\end{rem}
\begin{rem}{\rm We point out that by a proof completely similar to the one of Theorem \ref{CT}, one can prove the validity of the following generalized energy equality
\be\label{GEE}\ba{l}\displ\intl{\OO}|v(t)|^2\phi(t)dx+2\intll s t\intl
{\OO}|\nabla v|^2\phi\, dxd\tau= \intl{\OO}|v(s)|^2\phi(s)dx\VS\hskip1cm +\intll s t\intl{\OO}|v|^2(\phi_\tau+\Delta\phi)dxd\tau+\intll s t\intl{\OO}(|v|^2+2\pi_v)v\cdot\nabla\phi dxd\tau
-\widetilde H(t,s),\ea\ee
a.e. in $t\geq s>0$  and for $s=0$, where $$\widetilde H(t,s):=\left\{\hskip-0.2cm\ba{l}\displ\lim_{\alpha\to0} \lim_{m\to\infty}\alpha\!\intll st\! \frac{\|\phi^\frac 12(\tau)v^m(\tau)\dm_2^2}{\left(K+\|\nabla v^m(\tau)\|_2^2\right)^{\alpha+1}}\frac d{d\tau}\|\nabla v^m(\tau)\|_2^2\,d\tau,\mbox{ for  }s>0,\VS\lim_{s\to0}\lim_{\alpha\to0} \lim_{m\to\infty}\alpha\!\intll st \!\frac{\|\phi^\frac 12(\tau)v^m(\tau)\dm_2^2}{\left(K+\|\nabla v^m(\tau)\|_2^2\right)^{\alpha+1}}\frac d{d\tau}\|\nabla v^m(\tau)\|_2^2\,d\tau,\mbox{ for  }s=0\,,\ea\right.$$ 
for any arbitrary constant $K>0$ and for all 
nonnegative $\phi\in C_0^\infty(\R\times\ov\OO)$. }\end{rem}
\par The plan of the paper is the following. In Section\,2 we give some preliminaries and auxiliary lemmas. In Section\,3  we give the proof of the theorem. In the appendix we recall some known properties of the pressure field that are employed in Section\,2.
\section{\label{PR}Some preliminary results}For $p\in(1,\infty)$ we set $J^p(\OO):=$completion of $\mathscr C_0(\OO)$ in $L^p(\OO)$. By $P_p$ we denote the projector from $L^p(\OO)$ onto $J^p(\OO)$. In the case of $p=2$ we write $P_2\equiv P$. For any $R>0$ we set $B_R=\{x\in\R^3:|x|<R\}$.\par 
We start with the following a priori estimate:
\begin{lemma}\label{TINT}{\sl Let $\OO\subseteq\R^n$ and  let   $u\in W^{2,2}(\OO)\cap J^{1,2}(\OO) $. Then there exists a constant $c$ independent of $u$ such that    
\be\label{INTIII}\dm u\dm_r\!\leq\! c\dm P\Delta u\dm_2^a\dm u\dm_q^{1-a},\quad a\big(\mbox{$\frac12-\frac2n$}\big)+(1-a)\mbox{$\frac1q$}={\textstyle\frac1r},\ee
provided that   $a\in [0,1)$.}\end{lemma}
\begin{proof} The result of the lemma is a special case of a general one proved in \cite{MRIII,MRIV}.\end{proof}
\begin{lemma}[Friedrichs's lemma]\label{FR}{\sl Let $\OO$ be bounded. For all $\vep>0$ there exists  $N\in \N$ such that
\be\label{FRI}\dm u\dm_2^2\leq(1+\vep) \mbox{${\overset{N}{\underset{j=1}\sum}}$}(u,a^j)^2+\vep\dm\nabla u\dm_2^2,\mbox{ for any }u\in W^{1,2}(\OO)\,,\ee
where $\{a^j\}$ is an orthonormal basis of $L^2(\OO)$}\,.\end{lemma}
\begin{coro}\label{CFL}{\sl Assume that $\{u^k(t,x)\}$ is a sequence with 
\be\label{CFL-I}\intll0T \dm u^k(t)\dm_{{W^{1,2}(\OO)}}^2dt+\esssup_{(0,T)}\dm u^k(t)\dm_2\leq M<\infty, \mbox{ for all }k\in\N\,,\ee
and
\be\label{CFL-II}\ba{l}\dm u^k(t)\dm_{L^2(|x|>R)}^2\leq\dm u^k_0\dm_{L^2(|x|>\frac R2)}^2+c(t)\psi(R),\mbox{ for all }k\in\N,\VS\hskip3cm\mbox{ with }c(t)\in L^\infty((0,T)),  \mbox{ and }\lim_{R\to\infty}\psi(R)=0.\ea\ee
Also,  assume that 
\be\label{CFL-III}u^k_0\to u_0\mbox{ strongly in\,}L^2(\OO)\mbox{ and,  a.e. in }t\!\in\!(0,T),\,u^k(t)\to u(t)\mbox{ weakly in\,}L^2(\OO)\,.\ee   
Then there exists a subsequence of $\{u^k\}$ strongly converging to $u$ in $L^2(0,T;L^2(\OO))$\,.}\end{coro}\bp The result for $\OO$ bounded is well known, a proof is given in \cite{Ld}. In the case of $\OO$ exterior, a proof is due to Leray in \cite{L}. For the sake of the completeness we furnish the following proof.
\par
Let $u$ be the weak limit of $\{u^k\}$ in $L^2(0,T;L^2(\OO))$. 
By virtue of \rf{CFL-II}, for any $R>0$ we have
$$\ba{l}\dy\vs1\intll0T\intl{|x|>R}|u^k-u|^2\,dx\,dt\le\intll0T\intl{|x|>R}|u^k|^2+|u|^2\,dx\,dt\\
\hfill\dy\le \dm u^k_0\dm_{L^2(|x|>\frac R2)}^2+\psi(R)\intll0T c(t)\,dt+\intll0T\intl{|x|>R}|u |^2\,dx\,dt\VS\leq 2\big[\dm u^k_0-u_0\dm_{L^2(|x|>\frac R2)}+\dm u_0\dm_{L^2(|x|>\frac R2)}\big]+\psi(R)\intll0T c(t)\,dt+\intll0T\intl{|x|>R}|u |^2\,dx\,dt.\ea$$
By \eqref{CFL-II} for $\psi(R)$ and \rf{CFL-III} for $u_0^k$, and by the absolute continuity of the integral, we get that, for any $\varepsilon>0$, there exist $\ov R$ and $\ov k$, such that
$$\intll0T\intl{|x|>R}|u^k-u|^2\,dx\,dt<\varepsilon\,\mbox{ for all }R>\ov R\mbox{ and }k>\ov k\,.$$
In the bounded set $\Omega\cap B_{  R}$ we apply Lemma \ref{FR} and we use estimate \eqref{CFL-I}, obtaining, for any $k\in\N$,
\be\label{uk-uomegaR}\intll0T\intl{\Omega\cap B_{  R}}|u^k-u|^2\,dx\,dt
\le(1+\varepsilon)\sum_{j=1}^N\intll0T\left[\;\intl{\Omega\cap B_{  R}}(u^k-u)a^j\,dx\right]^2\,dt+2M\varepsilon.\ee
By the uniform bound \eqref{CFL-I} we have that
$$\left[\;\intl{\Omega\cap B_{ R}}(u^k-u)a^j\,dx\right]^2\le2M^2\|a^j\|_2^2,$$
and we use the dominated convergence theorem to pass to the limit as $k\to\infty$ in \eqref{uk-uomegaR}.
The property \eqref{CFL-III} allows us to complete the proof. 
\ep
We recall also two basic results. 
\begin{lemma}\label{infinitynorm}
Let $\Omega$ be a measurable subset of $\R^n$ and let $v\in L^q(\Omega)$ for any $q\ge\ov q\ge1$. If $\liminf\limits_{q\to\infty}\|v\|_q=l$ then $v\in L^\infty(\Omega)$ and $\|v\|_\infty=l$. 
\end{lemma}
\begin{proof}
There exists an increasing sequence $\{q_h\}$ such that $q_h\to\infty$ and
$\lim\limits_{h\to\infty}\|v\|_{q_h}=l$. Hence, for any $\varepsilon>0$ we can find $\bar h$ such that $\|v\|_{q_h}\le l+\varepsilon$ for any $h\ge\bar h$.
Moreover, if $q>q_{\bar h}$ we can find $h\ge \bar h$ such that $q_h\le q<q_{h+1}$. By interpolation, there exists $\theta_h\in[0,1]$ such that 
$$\|v\|_q\le \|v\|_{q_h}^{\theta_h}\|v\|_{q_{h+1}}^{1-\theta_h}\le l+\varepsilon.$$
It follows that $\|v\|_q\le l+\varepsilon$ for any $q>q_{\bar h}$. Hence $v\in L^\infty(\Omega)$ and
$$l=\liminf_{q\to\infty}\|v\|_q\le\limsup_{q\to\infty}\|v\|_q\le l+\varepsilon$$
for any $\varepsilon>0$. It follows that
$$\|v\|_\infty=\lim_{q\to\infty}\|v\|_q=l.$$
\end{proof}
\begin{lemma}\label{dominatedconv}
Let $\{g^k\}$ and $g$ be summable functions such that $g^k\to g$ almost everywhere and
$$\lim_{k\to\infty}\int g^k\,dx=\int g\,dx.$$
If $\{f^k\}$ and $f$ are measurable functions such that $|f^k|\le g^k$ almost everywhere and $f^k\to f$ almost everywhere, then
$$\lim_{k\to\infty}\int |f^k-f|\,dx=0.$$
\end{lemma}
\begin{proof} The result of lemma is contained in Theorem\,1.20 of \cite{EG}.
\end{proof}
 It is well known that in \cite{CKN} and in \cite{scheffer} it is furnished an existence theorem of suitable weak solutions to the Navier-Stokes Cauchy problem. Here, in order to achieve the same result in the case of problem \rf{NS}, that is, in the case of the initial boundary value problem in bounded or exterior domains $\OO$, we give the chief steps of the proof in Lemma\,\ref{LJAM} and in the Appendix. For this goal we consider  a mollified Navier-Stokes system. Hence problem \rf{NS} becomes
\be\label{MNS}\ba{l}v_t^m+J_m[v^m]\cdot
\nabla v^m+\nabla\pi_{v^m}=\Delta
v^m,\;\nabla\cdot
v^m=0,\mbox{ in }(0,T)\times\OO,\VS v^m=0\mbox{ on }(0,T)\times\po,\hskip0.12cm
v^m(0,x)=v_0^m(x)\mbox{ on
}\{0\}\times\OO,\ea\ee where $J_m[\cdot]$ is a mollifier and  $\{v_0^m\}\subset J^{1,2}(\OO)$ converges to $v_0$ in $J^2(\OO)$. The result of existence is established proving that the sequence of solutions $\{(v^m,\pi_{v^m})\}$ to problem \rf{MNS} converges with respect to the metric stated in Definition\,\ref{WS}, as well as proving that the limit satisfies the energy inequality \rf{SEI}. All this is a consequence of the following 
\begin{lemma}\label{LJAM}{\sl There exists a sequence  of solutions $\{(v^m,\pi_{v^m})\}$ such that, for all $m\in\N$ and $T>0$, $v^m\in C([0,T);J^{1,2}(\OO))\cap L^2(0,T;W^{2,2}(\OO))$. Moreover, for $\OO$ exterior domain, 
for $\ov R$ sufficiently large, we get
\be\label{LDD} \dm v^m(t)\dm_{L^2(|x|>R)}^2\leq  \dm v^m_0\dm_{L^2(|x|>\frac R2)}^2+ c(t)\psi(R)\,\mbox{ for any }t>0, \,R>2\ov R\mbox{ and }m\in\N\,,\ee 
with $c(t)\in L^\infty(0,T)$ and $\psi(R)=o(1)$. }\end{lemma}\bp The above result is well known. The existence and uniqueness of the solutions and related properties of regularity can be proved as in Theorem\,3 of \cite{Hy} (see also \cite{CF}). Concerning estimate \rf{LDD}, in the case of the Cauchy problem
 it was due   to Leray in \cite{L}. Subsequently  the result is extended to the initial boundary value problem in exterior domains by several authors, in different contexts. Actually, the technique employed by the authors  is essentially the same.   In this connection, without the aim of being exhaustive, we refer to \cite{GM,MS}. In Appendix we give  the details of the proof of \rf{LDD}.\ep 
\begin{lemma}\label{CL}{\sl For all $T>0$     the sequence of solutions to problem \rf{MNS} furnished by Lemma\,\ref{LJAM}, uniformly in $m\in\N$, satisfies the estimate
\be\label{CL-I} \left(\intll0{T}\left(\dm P\Delta v^m(t)\dm_2^2+\dm v^m_t(t)\dm_2^2\right)^{\frac13}dt\right)^3\leq c\left( \frac 1{ 1+\dm \n v^m(T)\dm_2^2}+ \dm v_0\dm_2^6\right) .\ee }
\end{lemma}
\begin{proof} 
By virtue of the regularity of $(v^m,\pi_{v^m})$ stated in Lemma\,\ref{LJAM}, we multiply equation \rf{MNS}$_1$ by $P\Delta v^m-v^m_t$. Integrating by parts on $\OO$, and applying the H\"older inequality, we get 
\be\label{PDelta-vt}\dm P\Delta v^m-v^m_t\dm_2\leq\dm v^m\cdot\n v^m\dm_2\,,\mbox{ a.e. in }t>0\,.\ee
Applying inequality \rf{INTIII}
with $r=\infty$ and $q=6$, by virtue of the Sobolev inequality, we obtain 
\be\label{NLT}\dm v^m\cdot\n v^m\dm_2\leq \dm v^m\dm_\infty\dm \n v^m\dm_2\leq c\dm  P\Delta v^m\dm_2^\frac12\dm \n v^m\dm_2^\frac32.\ee  
By inequalities \eqref{PDelta-vt} and \eqref{NLT}, we get 
\be\label{ddtnablavm-I}\ba{l}\dy\vs1\mbox{$\frac d{dt}$}\dm \n v^m\dm_2^2+\dm P\Delta v^m\dm_2^2+\dm v_t^m\dm_2^2 =\dm P\Delta v^m-v^m_t\dm_2^2\leq c\dm  P\Delta v^m\dm_2\dm \n v^m\dm_2^3\\
\hfill\dy\le\frac12\|P\Delta v^m\|_2^2+
c\dm \n v^m\dm_2^6,  \ea\ee
for all $m\in\N$ and a.e. in $t>0\,$.
We can divide by $(1+\dm \n v^m(t)\dm_2^2)^2$, and  the following holds
$$\frac{\mbox{$\frac d{dt}$}\dm \n v^m\dm_2^2}{(1+\dm \n v^m\dm_2^2)^2\hskip-0.1cm\null}\hskip0.1cm+\frac{\frac12\dm P\Delta v^m\dm_2^2+\dm v_t^m\dm_2^2}{(1+ \dm \n v^m\dm_2^2)^2}\leq c\dm \n v^m\dm_2^2\,.$$ 
Integrating on $(0,T)$,   we have
$$\frac1{1+ \|\nabla v^m_0\|_2^2}-\frac1{1+ \|\nabla v^m(T)\|^2_2}+\int_0^T\frac{\frac12\dm P\Delta v^m\dm_2^2+\dm v_t^m\dm_2^2}{(1+ \dm \n v^m\dm_2^2)^2}\,dt\le c \int_0^T\|\nabla v^m\|_2^2\,dt.
$$
Employing  the reverse H\"older inequality (see \cite[Theorem 2.12]{AF}) with exponents $\frac13$ and $-\frac12$, we get
$$\int\limits_0^T\frac{\frac12\dm P\Delta v^m\dm_2^2+\dm v_t^m\dm_2^2}{(1+ \dm \n v^m\dm_2^2)^2}\,dt\ge \Big[\int\limits_0^T\big[\mbox{$\frac12$}\|P\Delta v^m\|_2^2+\|v_t^m\|_2^2\big]^{\frac13}dt\Big]^{\hskip-1pt 3}\Big[\int\limits_0^T (1+\|\nabla v^m\|_2^2)  dt\Big]^{\hskip-1pt -2}\!.$$
Coupling the above inequalities with the energy inequality \rf{EI}, estimate \rf{CL-I} follows.  
\end{proof}
\section{Proof of Theorem\,\ref{CT}} The idea of the proof is the following. We consider the sequence of solutions to problem \rf{MNS} furnished by Lemma\,\ref{LJAM}. It is well known that there exists a subsequence $\{(v^m,\pi_{v^m})\}$ whose weak limit   $(v,\pi_v)$ in $L^2(0,T;J^{1,2}(\OO))$ is a weak solution in the sense of  Definition\,\ref{SWS}. All this is contained in \cite{L} or, for example, also in \cite{CKN}. Now, our aim is to prove further estimates on the extract $\{(v^m,\pi_{v^m})\}$ that ensure the thesis of Theorem\,\ref{CT}. \subsection{The strong convergence in $L^p(0,T;L^2(\OO))$  for all $p\in[1,2)$}   We start by proving that the sequence $\{v^m\}$ strongly converges in $L^p(0,T;W^{1,2}(\OO))$, for $p\in [1,2)$ and for all $T>0$. We recall that    
$$\dm \n u\dm_2\leq \dm P\Delta u\dm_2^\frac12\dm u\dm_2^\frac12\,, \mbox{ for all }u\in W^{2,2}(\OO)\cap J^{1,2}(\OO)\,.$$
Hence, integrating on $(0,T)$ and applying the H\"older inequality, we get 
$$\intll0T\dm \n v^k(t)-\n v^m(t)\dm_2\,dt\leq \Big[\intll0T\dm P\Delta v^k(t)-P\Delta v^m(t)\dm_2^\frac23dt\Big]^\frac34\Big[\intll0T\dm v^k(t)-v^m(t)\dm_2^2\,dt\Big]^\frac14\,.$$ 
By virtue of Lemma\,\ref{CL},    we get the existence of a $M(T)$ such that
$$\intll0T\dm \n v^k(t)-\n v^m(t)\dm_2\,dt\leq (2M(T))^\frac34\Big[\intll0T\dm v^k(t)-v^m(t)\dm_2^2\,dt\Big]^\frac14\,, \mbox{ for all, } k,m\in\N\,,$$
and, via \rf{LDD}, we can apply Corollary\,\ref{CFL} to deduce the strong convergence of the sequence $\{v^m\}$ in $L^1(0,T;W^{1,2}(\OO))$. Since the energy inequality holds uniformly with respect to $m\in\N$, by interpolation we arrive at the strong convergence of $\{v^m\}$ in $L^p(0,T;W^{1,2}(\OO))$, for any $p\in [1,2)$. 
In order to identify the limit point, we remark that $\{v^m\}$ weakly converges to $v$ in $L^2\left(0,T;W^{1,2}(\Omega)\right)$, hence $v$ has to coincide with the strong limit in each space $L^p\left(0,T;W^{1,2}(\Omega)\right)$.
Thus for all $T>0$  we deduce that  
\be\label{Clim}\ba{ll}\displ\intll st\dm \n v(\tau)\dm_2^2d\tau\hskip-0.2cm&\displ=\lim_{p\to2^-}\intll st\dm\n v(\tau)\dm_2^pd\tau\\&\displ=\lim_{p\to2^-}\lim_{m\to\infty}\intll st\dm \n v^m(\tau)\dm_2^pd\tau\,,\mbox{ for all }t,s\in(0,T) \,.\ea\ee \subsection{Proof of formula \rf{EE-I}}
By the strong convergence in $L^1\left(0,T;W^{1,2}(\Omega)\right)$, there exists a negligible set (for the Lebesgue measure) $\mathcal I\subset (0,T)$, such that 
for any $t\in\mathcal G_1:=(0,T)-\mathcal I$ the following limits are finite 
\be\label{Lpun}\lim_{m\to\infty}\dm  v^m(t)\dm_2=\dm  v(t)\dm_2\mbox{\; and \;}\lim_{m\to\infty}\dm \n v^m(t)\dm_2=\dm \n v(t)\dm_2\,.\ee
From the energy equality for the approximating solutions $\{v^m\}$ we obtain, for any $t\in \mathcal G_1$ and any $\alpha,K>0$,
\be\label{FS}\frac 1{\left(K+\dm \n v^m(t)\dm_2^2\right)^\alpha\hskip-0.1cm\null}\hskip0.1cm\frac d{dt}\dm v^m(t)\dm_2^2+\frac{2\dm \n v^m(t)\dm_2^2}{\left(K+\|\nabla v^m(t)\|_2^2\right)^\alpha}=0\,.\ee
Integrating by parts we get
$$\ba{l}\dy\vs1\alpha\intll st\frac{\dm v^m(\tau)\dm_2^2}{\left(K+\dm\n v^m(\tau)\dm_2^2\right)^{\alpha+1}}\hskip0.3cm\frac d{d\tau} {\dm\n v^m(\tau)\dm_2^2\hskip-0.2cm}\hskip0.2cmd\tau+2\intll st\frac{\dm \n v^m(\tau)\dm_2^2}{\left(K+\|\nabla v^m(\tau)\|_2^2\right)^\alpha}\,d\tau\\
\hfill\dy= \frac{\dm v^m(s)\dm_2^2}{\left(K+\dm\n v^m(s)\dm_2^2\right)^\alpha\hskip-0.2cm\null}\hskip0.1cm-\frac{\dm v^m(t)\dm_2^2}{\left(K+\dm\n v^m(t)\dm_2^2\right)^\alpha\hskip-0.2cm}\hskip0.2cm\,.\ea$$ 
We remark that, for almost every $\tau\in(0,T)$, by  \eqref{Lpun},
$$\frac{\|\nabla v(\tau)\|_2^2}{\left(K+\|\nabla v(\tau)\|_2^2\right)^\alpha}\leftarrow\frac{\|\nabla v^m(\tau)\|_2^2}{\left(K+\|\nabla v^m(\tau)\|_2^2\right)^\alpha}\le \|\nabla v^m(\tau)\|_2^{2-2\alpha}\rightarrow\|\nabla v(\tau)\|_2^{2-2\alpha}$$
and that, for $\alpha\in\left(0,\frac12\right]$, by virtue of the strong convergence in $L^{2-2\alpha}\left(0,T;W^{1,2}(\Omega)\right)$,
$$\lim_{m\to\infty}\int_s^t\|\nabla v^m(\tau)\|_2^{2-2\alpha}\,d\tau=\int_s^t\|\nabla v(\tau)\|_2^{2-2\alpha}\,d\tau.$$
Hence we can apply Lemma \ref{dominatedconv} to obtain that, for any $t,s\in \mathcal G_1$,
\be\label{beforelimalpha} \ba{l}\dy\vs1
\lim_{m\to\infty}\alpha\intll st\frac{\|v^m(\tau)\dm_2^2}{\left(K+\|\nabla v^m(\tau)\|_2^2\right)^{\alpha+1}}\frac d{d\tau}\|\nabla v^m(\tau)\|_2^2\,d\tau+2\intll st\frac{\dm \n v (\tau)\dm_2^2}{\left(K+\|\nabla v(\tau)\|_2^2\right)^\alpha}\,d\tau\\
\hfill\dy= \frac{\dm v(s)\dm_2^2}{\left(K+\dm\n v(s)\dm_2^2\right)^\alpha\hskip-0.2cm\null}\hskip0.1cm-\frac{\dm v(t)\dm_2^2}{\left(K+\dm\n v(t)\dm_2^2\right)^\alpha\hskip-0.2cm}\hskip0.2cm\,. \ea
\ee
Applying once again Lemma \ref{dominatedconv}, we get
$$\lim_{\alpha\to0}\intll st\frac{\dm \n v (\tau)\dm_2^2}{\left(K+\|\nabla v(\tau)\|_2^2\right)^\alpha\hskip-0.2cm}\hskip0.2cm\,d\tau=\intll st\dm \n v (\tau)\dm_2^2\,d\tau.$$
Then, letting $\alpha\to0$ in \eqref{beforelimalpha}, we deduce \rf{EE-I} with 
\be\label{CPEE}H(t,s):=\lim_{\alpha\to0} \lim_{m\to\infty}\alpha\intll st \frac{\|v^m(\tau)\dm_2^2}{\left(K+\|\nabla v^m(\tau)\|_2^2\right)^{\alpha+1}}\frac d{d\tau}\|\nabla v^m(\tau)\|_2^2\,d\tau\,.\ee\subsection{Proof of formula \rf{EE-II}} We denote by $\mathcal G_2$ the set of $t\geq0$ such that the estract $\{v^m\}$ is strongly convergent in $L^2(\OO)$. Recalling the definition of $\mathcal G_1$, we have $\mathcal G_1\subseteq\mathcal G_2$. By virtue of Lemma\,\ref{LA} we claim that $\dm \n v^m(t)\dm_2\ne0$ for all $t>0$ and $m\in \N$. Hence we consider formula \rf{FS} that rewrite with $K_1$ $$\frac 1{\left(K_1+\dm \n v^m(t)\dm_2^2\right)^\alpha\hskip-0.2cm\null}\hskip0.2cm\frac d{dt}\dm v^m(t)\dm_2^2+\frac{2\dm \n v^m(t)\dm_2^2}{\left(K_1+\|\nabla v^m(t)\|_2^2\right)^\alpha\hskip-0.2cm}\hskip0.2cm=0\,,$$ where, by  the above claim, we can consider  $K_1\geq0$. Integrating on $(s,t)$, for $s,t\in \mathcal G_2$, and applying the mean value theorem for the integrals, we get $$\frac 1{\big(K_1+\dm \n v(t_{\alpha,m})\dm_2^2\big)^\alpha\hskip-0.2cm}\hskip0.2cm= \big(\dm v^m(s)\dm_2^2-\dm v^m(t)\dm_2^2\big)^{-1} 2\intll st\frac{\dm \n v^m(\tau)\dm_2^2}{\left(K_1+\|\nabla v^m(t)\|_2^2\right)^\alpha\hskip-0.2cm}\hskip0.2cmd\tau\,.$$ Since the right hand side admits limit as $m\to\infty$ and as $\alpha\to0$,   the limit  $F(t,s):=\displ\lim_{\alpha\to0}\lim_{m\to\infty}\frac 1{\left(K_1+\|\nabla v^m(t_{\alpha,m})\|_2^2\right)^\alpha\hskip-0.2cm}\hskip0.2cm$  is well posed and \rf{EE-II} is proved.
\subsection{The $L^{\mu(q)}(0,T;L^q(\OO))$ property} 
By virtue of estimate \rf{INTIII} we get
$$\dm v^m(t)\dm_\infty\leq c\dm P\Delta v^m\dm_2^\frac12\dm\n v^m\dm_2^\frac12\,.$$ Employing the energy relation \rf{EI} and estimate \rf{CL-I}, applying  H\"older's  inequality,  for all $T>0$, we deduce that, for any $m\in\N$,
$$\intll0T\dm v^m(\tau)\dm_\infty d\tau\leq c\Big[\intll0T\dm P\Delta v^m(\tau)\dm_2^\frac23d\tau\Big]^\frac34\Big[\intll0T\dm \n v^m(\tau)\dm_2^2d\tau\Big]^\frac14\leq C(v_0),$$ 
where, here and in the following, $C(v_0)$ are constants depending only on $\|v_0\|_2$.
Therefore, by $L^p$-interpolation and recalling that $v^m\in L^\infty\left(0,T;J^2(\OO)\right)$, uniformly in $m\in\N$ and $T>0$, we arrive at
$$\intll0T\dm v^m(\tau)\dm_q^\frac q{q-2}d\tau\leq \sup_{(0,T)}\dm v^m(\tau)\dm_2^\frac2{q-2} \intll0t \dm v^m(\tau)\dm_\infty d\tau\leq C(v_0).$$
This allows us to claim that, for all $T>0$, the weak solution $v$ to problem \rf{NS}, limit of the sequence   $\{v^m\}$,    belongs to $L^\frac q{q-2}(0,T;L^q(\OO))$, for all $q\in(6,\infty)$, with 
$$\intll0T\dm v(\tau)\dm_q^\frac q{q-2}d\tau\leq  C(v_0).$$
This limit property and Fatou's lemma  ensure that, for all $T>0$ and for any sequence $q_h\to \infty$, the following estimate holds true
\be\label{mL}\intll0T{\liminf_{h\to\infty}}\dm v(\tau)\dm_{q_h}\leq	\lim_{h\to\infty} T^\frac2{q_h}C(\dm v_0\dm_2)^\frac{q_h-2}{q_h}=C(v_0)\,.\ee
The thesis of the theorem in the case $q=\infty$ follows straightforward by Lemma \ref{infinitynorm}.\chiu
\begin{rem}\label{2D}{\rm We verify that $H(t,s)=0$ in the case of $2D$-Navier-Stokes equations. It is important to realize the result in the framework of the construction given in the above proof, that is, not relying on the regularity of the limit  solution $v$. We start remarking that estimate \rf{NLT}, for $\Omega\subset\R^2$, via \rf{INTIII}, becomes 
$$\dm v^m\cdot\n v^m\dm_2\leq c\dm v^m\dm_2^\frac12\dm P\Delta v^m\dm_2^\frac12\dm \n v^m\dm_2\,.$$
Hence, in place of \eqref{ddtnablavm-I}, we deduce the differential inequality
\be\label{DD}\frac d{dt}\dm \n v^m\dm_2^2+\dm P\Delta v^m\dm_2^2+\dm v^m_t\dm_2^2\leq c\dm v^m\dm_2^2\dm \n v^m\dm_2^4\leq c\dm  v^m_0\dm_2^2\dm \n v^m\dm_2^4\,.\ee
Hence we achieve the result of Lemma \ref{CL} also for $\Omega\subset\R^2$, with the only difference that on the right hand side of estimate \eqref{CL-I} we have $\frac1{(1+\|\nabla v^m(T)\|_2^2)}+c\|v_0\|_2^2$.
By the same arguments of the three-dimensional case, we obtain that $\{v^m\}$ strongly converges in $L^p(0,T;W^{1,2}(\OO))$, for all $p\in [1,2)$, that is the key ingredient to arrive at the identity \rf{EE-I}. 
\par
Now we prove that $H(t,s)\le0$, which implies, by virtue of the energy inequality \rf{EI}, that $H(t,s)=0$. 
By \eqref{DD}, \rf{EI} and the H\"older inequality, we have
$$\ba{l}\dy\vs1\alpha\int_s^t\frac{\|v^m(\tau)\|_2^2}{\left(K+\dm \n v^m(\tau)\dm_2^2\right)^{\alpha+1}}\hskip0.1cm\frac d{d\tau}\dm \n v^m(\tau)\dm_2^2d\tau\\
\hfill\dy\vs1\le\alpha c\|v_0^m\|_2^2\int_s^t\|v^m(\tau)\|_2^2\frac{\|\nabla v^m\|_2^4}{\left(K+\|\nabla v^m\|_2^2\right)^{\alpha+1}}\,d\tau\le\alpha c\|v_0^m\|_2^4\int_s^t\|\nabla v^m(\tau)\|_2^{2-2\alpha}\,d\tau\\
\hfill\dy\le\alpha c\|v_0^m\|_2^4 (t-s)^{\frac1\alpha}\left(\int_s^t\|\nabla v^m(\tau)\|_2^2\,d\tau\right)^{1-\alpha}\le \alpha c\|v_0^m\|_2^{6-2\alpha}(t-s)^{\frac1\alpha}.
\ea$$
Passing to the limit for $m\to\infty$ and then for $\alpha\to0$, we get that $H(t,s)\le0$.}
\end{rem}\section{Appendix}
\subsection{\label{AX-I}Some results related to the construction of the weak solution}In this section we recall some results which are fundamental in order to construct a suitable weak solution. These results essentially concern estimates of the pressure field $\pi_{v^m}$ which appears in \rf{MNS}. Of course we look for estimates that are uniform with respect to $m\in\N$. Our aim is to justify estimate \eqref{LDD}.\par We start  by recalling that the energy relation holds uniformly in $m\in\N$:
$$\dm v^m(t)\dm_2^2+2\intll0t\dm\n v^m(\tau)\dm_2^2d\tau=\dm v^m_0\dm_2^2\leq\dm v_0\dm_2^2\,\mbox{ for all }t>0\,.$$
\par We introduce the following functionals:
$$\ba{ll}\lambda\in(0,1),\;q\in(1,\infty),\;\displ<a>_q^\lambda&\hskip-0.2cm\displ:=\Big[\intl{\po}\intl{\po}\frac{|a(x)-a(y)|^q}{|x-y|^{2+\lambda q}}d\sigma_yd\sigma_x\Big]^{\frac1q}\,,\VS \hskip3.2cm \dm a\dm_{W^{1-\frac1q,q}(\po)}&\hskip-0.2cm:=\dm a\dm_{L^q(\po)}+<a>_q^{1-\frac1q}\,.\ea$$ 
We consider the following Neumann problem: \be\label{NP}\Delta \pi=0,\;\pi\to0\mbox{ for }|x|\to\infty, \;\frac {d\pi}{d\nu}=\nu\cdot\n\times{a}\mbox{ on }\po\,.\ee \begin{lemma}\label{LA-I}{\sl In \rf{NP} assume $a\in W^{1-\frac1q,q}(\po)$. Then for all $\lambda\in(0,1-\mbox{$\frac1q$}]$ and $R_0$ sufficiently large there exists a constant $c$ independent of $a$ such that
\be\label{NP-I}  \dm \pi\dm_{L^q(\OO\cap B_{R_0})}\leq c<a>_q^{\lambda}\,.\ee}\end{lemma} The lemma is due to Solonnikov in \cite{Sl-R,Sl-E}. A recent proof of the same result, by similar techniques,  can be found, for example, in \cite{MR-V}. \par Applying the H\"older inequality and the Gagliardo trace theorem, one gets
\be\label{NP-II}\ba{ll}\dm \pi\dm_{L^q(\OO\cap B_{R_0})}&\hskip-0.2cm\leq c<a>_q^{\lambda}\leq c\dm a\dm_{L^q(\po)}^\beta\Big[<a>_q^{1-\frac1q}\Big]^{1-\beta}\VSE\leq c \Big[\dm a\dm_{L^q(\OO\cap B_{R_0})}+\dm a\dm_{L^q(\OO)}^\frac1{q'\hskip-0.1cm}\dm \n a\dm_q^\frac1q\Big]^\beta   \dm \n a\dm_q^ {1-\beta} \ea\ee with $\beta:=\frac{q(1-\lambda)-1}{1+q}$\,. Now, we consider $(U,\pi)$ as a solution to the Stokes problem \be\label{ST}\ba{l}U_t+\n\pi=\Delta U\,,\;\n\cdot U=0\,,\; \mbox{ in }(0,T)\times\OO\,,\VS U=0\mbox{ on }(0,T)\times\po\,, \;U=v_0\mbox{ on }\{0\}\times\OO\,.\ea\ee We estimate   $\pi$ by means of \rf{NP-II}. That is, we set  $a:=\mbox{curl} \hskip0.05cmU$, we assume $v_0\in J^2(\OO)$, and,
  via the semigroup properties of $U$ (see, e.g., \cite{MS-III}),  for $q=2$,   for all $T>0$, we get
\be\label{NP-III}\dm \pi(t)\dm_{L^2(\OO\cap B_{R_0})}\leq c(T)\dm v_0\dm_2\Big[t^{-1+\frac\beta2}+t^{-1+\frac\beta4}\Big]\,, \mbox{ for all }t\in (0,T)\,,\ee with $\beta=\frac{1-2\lambda}{3}$. Keeping this in hand, we can also deduce an estimate in the exterior of $B_{R_0}$. Actually, by means of a cut of the equation \rf{NP} in $B_{R_0}$, we get
$$\Delta(\pi h_{R_0})=\pi\Delta h_{R_0}+2\nabla\pi\cdot\nabla h_{R_0},$$
with $h_{R_0}$ smooth function such that $h_{R_0}(x)=1$ for $|x|>R_0$, $h_{R_0}(x)=0$ for $|x|<\frac {R_0}{2}.$ 
Then, by the representation formula of the solution, we obtain  
$$\pi(t,x)= -\intl{\R^3}\mathcal E(x-y)\pi\Delta h_{R_0} dy-2\intl{\R^3}\nabla\mathcal E(x-y)\nabla h_{R_0}\pi dy,$$
with $\mathcal E$ fundamental solution. So that, for $\ov r>3$ and for $|x|>2R_0$, we  easily get \be\label{NPIV}\ba{ll} \dm \pi(t)\dm_{L^{\ov r}(|x|>R_0)}\hskip-0.2cm&\leq c\dm \pi(t)\dm_{L^2(\OO\cap B_{R_0})}\VS&\leq c(T)\dm v_0\dm_2\Big[t^{-1+\frac\beta2}+t^{-1+\frac\beta4}\Big],\mbox{ for all }t\in(0,T)\,.\ea\ee
\par Consider the following initial boundary value problem for the Stokes system:
\be\label{STF}\ba{l}W_t-\Delta W+\n\pi_W=F\,,\;\n\cdot W=0\,,\;\mbox{ on }(0,T)\times\OO\,,\VS W=0\mbox{ on  }(0,T)\times\OO\,,\;W=0\mbox{ on }\{0\}\times\OO\,.\ea\ee \begin{lemma}\label{EXF}{\sl In problem \rf{STF} assume $F\in L^r(0,T;L^s(\OO)),\, \frac3s+\frac2r=4,\,s\in(1,\frac32)$. Then there exists a unique solution to problem \rf{STF} such that \be\label{EXF}\intll0T\Big[\dm D^2W(\tau)\dm_s^r+\dm \n \pi_W(\tau)\dm_s^r+\dm W_\tau(\tau)\dm_s^r\Big]d\tau\leq c\intll0T\dm F(\tau)\dm_s^rd\tau\,,\ee
with $c$ independent of $F$ and $T$.}\end{lemma} \bp This result is well known, a proof can be found in \cite{MS-III,MS-IV}.\ep\begin{lemma}\label{APWS}{\sl Let $\{(v^m,\pi_{v^m})\}$ be the   sequence of solutions   to problem \rf{MNS} furnished by Lemma\,\ref{LJAM}. Then there exist functions $\pi^1_{v^m},\,\pi^2_{v^m}$ such that   $\pi_{v^m}=\pi^1_{v^m}+\pi^2_{v^m}$, and, for all $\ov r>3$, $R_0>0$ and $\lambda\in (0,\frac12)$, we also obatin \be\label{APWS-I}\ba{l} \dm \pi^1_{v^m}(t)\dm_{L^2(\OO\cap B_{R_0})}+\dm \pi_{v^m}^1(t)\dm_{L^{\ov r}(|x|>R_0)}\leq c(T)\dm v_0\dm_2t^{-1+\frac \beta4}\,,\mbox{ with $\beta:=\frac{1-2\lambda}3$}\,,\\ \hskip-0.8cm\mbox{and}\\{\displ\intll0T\dm \n\pi_{v^m}^2(\tau)\dm_s^rd\tau\leq c\dm v_0\dm_2^{2r}}\,,\ \frac 3s+\frac2r=4\,.\ea\ee}\end{lemma}\bp The result of the lemma is an immediate consequence of the following decomposition:$$\mbox{for all }m\in\N, v^m=U^m+W^m\,\;\mbox{ and }\;\pi_{v^m}=\pi_{U^m}+\pi_{W_m}\,,$$ with $(U_m,\pi_{U^m})$ solution to problem \rf{ST} with initial data $U^m=v^m_0$, and $(W^m,\pi_{W^m})$ solution to problem \rf{STF} with $F^m=J_m(v^m)\cdot \n v^m$. Since $v^m\in L^2(0,T;J^{1,2}(\OO))$, one easily deduces that $F^m\in L^r(0,T;L^s(\OO))$ provided that $\frac 3s+\frac2r=4$. \ep From estimate \rf{APWS-I}$_2$, via the Sobolev embedding theorem (cf. Lemma\,5.2 of \cite{GPG}), there exists a function $\pi_{v^m}^0(\tau)$ such that \be\label{APWS-II}\intll0T\dm \pi_{v^m}^2(\tau)-\pi_{v^m}^0(\tau)\dm_{\frac{3s}{3-s}}^rd\tau\leq c
\intll0T\dm \nabla \pi_{v^m}^2(\tau)\dm_{s}^rd\tau
\leq c\dm v_0\dm_2^{2r}\,,\ee for $\frac 3s+\frac2r=4$. In the sequel we assume that $\pi_{v^m}:=\pi_{v^m}^1(t,x)+\pi_{v^m}^2(t,x)-\pi_{v^m}^0(t)$.\par
Now, we are in a position to prove estimate \rf{LDD}.\par We consider $R>0$ such that $  R >2R_0$. We denote by $h_R$ a smooth function such that $h_R(x)=1$ for $|x|>R$, $h_R=0$ for $|x|<\frac R2$ with  $|D^2 h_R|+|\n h_R|\leq \frac cR$. We multiply equation \rf{MNS}$_1$ by $v^m(t,x) h_R(x)$. Integrating by parts on $(0,T)\times\OO$, we get
$$\ba{ll}\dm v^m(t)\dm_{L^2(|x|>R)}^2\hskip-0.2cm&\displ\leq \dm v^m_0\dm_{L^2(|x|>\frac R2)}^2+ cR^{-1}\!\intll0t\!\Big[\dm v^m \dm_2^2+\dm v^m\dm_3^3+\dm \pi_{v^m}|v^m|\dm_{L^1(R<|x|<2R)}\Big]d\tau\VS&\displ=:\dm v^m_0\dm_{L^2(|x|>\frac R2)}^2+ cR^{-1}\!\intll0t\!\big[I_1(\tau)+I_2(\tau)+I_3(\tau)\big]d\tau.\ea$$ Applying the Gagliardo-Nirenberg inequality and the energy relation, we deduce
$$I_1(\tau)+I_2(\tau)\leq \dm v^m_0\dm_2^2+c\dm v^m_0\dm_2^\frac32\dm\n v^m(\tau)\dm_2^\frac32\,.$$ By virtue of Lemma\,\ref{APWS},     assuming $\ov r\in(3,6)$ and $\frac 3s+\frac2r=4$, by applying the H\"older inequality,  for $I_3$ we get $$\ba{ll}\dm \pi_{v^m}|v^m|\dm_{L^1(R<|x|<2R)}\hskip-0.2cm&\leq \Big[cR^{3\frac{\ov r-2}{2\ov r}}\dm \pi_{v^m}^1(\tau)\dm_{L^{\ov r}(|x|>\frac R2)} +cR^\frac {5s-6}{2s}\dm \pi^2_{v^m}-\pi^0_{v^m}\dm_\frac{3s}{3-s}\Big]\dm v^m\dm_2\VS&\leq c\Big[R^{3\frac{\ov r-2}{2\ov r}}\dm \pi_{v^m}^1(\tau)\dm_{L^{\ov r}(|x|>\frac R2)} +R^\frac {5s-6}{2s}\dm \n\pi^2_{v^m} \dm_s\Big]\,,\ea $$ for all $R>2R_0$.
Increasing the terms $I_i$, $i=1,2,3$, by means of the above estimates,  we arrive at 
$$\ba{ll}\dm v^m(t)\dm_{L^2(|x|>R)}^2\hskip-0.2cm&\displ\leq\dm v^m_0\dm_{L^2(|x|>\frac R2)}^2+c R^{-1}\intll0t\Big[\dm v^m_0\dm_2^2+c\dm v^m_0\dm_2^\frac32\dm\n v^m(\tau)\dm_2^\frac32\VS&\hskip 4cm+ R^{3\frac{\ov r-2}{2\ov r}}\dm \pi_{v^m}^1(\tau)\dm_{L^{\ov r}(|x|>\frac R2)}+R^\frac {5s-6}{2s}\dm \n\pi^2_{v^m} \dm_s\Big]d\tau\,.\ea$$  Via estimates \rf{APWS-I}, applying the H\"older inequality and the energy relation we prove that $$ \dm v^m(t)\dm_{L^2(|x|>R)}^2   \leq \dm v^m_0\dm_{L^2(|x|>\frac R2)}^2+  cR^{-1}\big[ t+t^\frac14\dm v_0^m\dm_2+t^{\frac\beta4}R^{3\frac{\ov r-2}{2\ov r}}+t^{1-\frac1r}\dm v^m_0\dm_2R^{\frac{5s-6}{2s}}\big]\dm v^m_0\dm_2^2\,, $$
which furnishes \rf{LDD}.\subsection{\label{AX-II}Uniqueness backward in time for the sequence $\{v^m\}$} For the sake of the completeness we prove a  result concerning the uniqueness backward in time for solutions $(v^m,\pi^m)$ to the IBVP \rf{MNS}, whose existence is furnished by Lemma\,\ref{LJAM}.  A wide literature on the topic can be found in \cite{Py} and \cite{AS}. Here we employ   the logarithmic convexity method developed in \cite{KP,GS}. In order to prove the result we premise a result\begin{lemma}\label{LI}{Let $v^m_0\in \mathscr C_0(\OO)$ and $(v^m,\pi_{v^m})$ the solution to problem \rf{MNS}. Then, for all $T>0$, there exists a constant $A_m$ such that  \be\label{LI-I}\dm v^m(t)\dm_\infty\leq A_m\,\mbox{ for all }t\in[0,T]\,.\ee}\end{lemma} \bp The result of the lemma is classical in the case of a solution to problem \rf{NS}, provided that one considers $[0,T]$ as a subset of the local interval of  existence of the solution. In the case of a solution to  problem \rf{MNS}, by employing the properties of the mollifier, one can prove property \rf{LI-I} on $[0,T]$ for all $T>0$ with a bound depending on $m$. Actually, for all $T>0$, one proves Ladyzhenskaya's estimate (see \cite{Ld} or \cite{Hy}), that is $\dm v^m(t)\dm_{2,2}\leq A_m$ for any $t\in [0,T]$. These considerations  allow us to omit further details related to estimate \rf{LI-I}.\ep
We are going to prove\begin{lemma}\label{LA}{\sl If $v_0^m\ne0$, then the solution $(v^m,\pi^m)$ to problem \rf{MNS} enjoys the property    $\dm\n v^m(t)\dm_2>0$ for all $t>0$\,.}\end{lemma}\bp
 We start from \rf{PDelta-vt}, that furnishes:
\be\label{A-I}\mbox{$\frac d{dt}$}\dm \n v^m\dm_2^2+ \dm P\Delta v^m\dm_2^2+\dm v_t^m\dm_2^2\leq \dm v\dm_\infty^2\dm \n v^m\dm_2^2\,,\mbox{  a.e. in }t\!>\!0\,.\ee We recall that the following estimates hold:
\be\label{A-II}\dm \n v^m\dm_2^2\leq \dm P\Delta v^m\dm_2\dm v^m\dm_2\mbox{\; and \;}\dm \n v^m\dm_2^2\leq \dm  v^m_t\dm_2\dm v^m\dm_2\,.\ee By virtue of the energy equation  for $(v^m,\pi^m)$, we get
\be\label{A-III}\dot E(t)=-2\dm \n v^m\dm_2^2\,, \; \ddot E(t)=-2\frac d{dt}\dm \n v^m\dm_2^2\,,\ee where we set $E(t):=\dm v^m\dm_2^2$. Therefore, by \rf{LI-I} and \rf{A-I}-\rf{A-III} we arrive at
\be\label{A-IV}- \ddot E+\frac{{\dot E}^2\hskip-0.1cm}{E}\leq 2A_m^2\dm \n v^m\dm_2^2\,.\ee 
We prove the result of the lemma claiming that if  $\ov t>0$ is the first instant  such that $\dm\n v^m(\ov t)\dm_2=0$, then $\dm v^m(t)\dm_2=0$ for all $t\in [0,\ov t]$, that is a contradiction. Since, for all $T>0$, $ v^m\in C([0,T);J^{1,2}(\OO))$, if  there exists $\ov t>0$ such that $\dm  v^m(\ov t)\dm_2=0$, then there exists $\delta>0$  such that   $\dm v^m(t)\dm_2\leq 1$ for all $t\in [\ov t-\delta,\ov t]$.   So that, for a suitable $h>0$, the inequality \rf{A-IV} can be written as $$\ddot E-\frac{{\dot E}^2}E\geq  h {{\dot E}}\;\Rightarrow \;\frac{\ddot EE-{\dot E}^2}{E^2}\geq h \frac{{\dot E}}E\Leftrightarrow \frac d{dt}\Big[e^{-ht}\frac{\dot E}E\Big]\geq0\,,\mbox{ for all }\;t\in[\ov t-\delta,\ov t]\,.$$ Set $\sigma=e^{-ht}$, we deduce
$$\frac {d}{d\sigma}\Big[\frac1 E \frac d{d\sigma}E\Big]\geq0\,.$$
 This last  implies that $\log E$ is a convex function.   That is,   for all $h\in [0,1]$,
\be\label{log}\ba{l} \log{E(ht+(1-h)t_0)}\leq h \log{E(t)}+(1-h) \log{E(\overline t-\delta)} \,.\ea\ee Since in $\ov t>0$ it is $\dm   v^m(\ov t)\dm_2=0$, then we arrive at $\dm   v^m(t)\dm_2=0$ for all  $t<\ov t$, which is a contradiction with the hypothes   $v_0^m\ne0$.  \ep

\end{document}